\documentclass[11pt]{amsart}
\usepackage{amsfonts}

\newtheorem{thm}{Theorem}[section]
\newtheorem{lem}[thm]{Lemma}
\newtheorem{cor}[thm]{Corollary}
\newtheorem{prop}[thm]{Proposition}

\theoremstyle{definition}
\newtheorem{exmp}[thm]{Example}
\newtheorem{defn}[thm]{Definition}

\theoremstyle{remark}
\newtheorem{rem}[thm]{Remark}

\newcommand{\Ext}{\operatorname{Ext}}
\newcommand{\Hom}{\operatorname{Hom}}

\newcommand{\add}{\operatorname{add} }
\newcommand{\pd}{\operatorname{pd} }
\newcommand{\rep}{\operatorname{rep} }
\newcommand{\ind}{\operatorname{ind} }
\newcommand{\Gen}{\operatorname{Gen}}
\newcommand{\Cogen}{\operatorname{Cogen}}
\newcommand{\modu}{\operatorname{mod}\text{-} }
\newcommand{\dimvec}{\underline{\operatorname{dim} }}
 
\begin{document}
\title[The number of arrows in the quiver of tilting modules]{The number of arrows in the quiver of tilting modules over a path 
algebra of type $A$ and $D$}
\author{Ryoichi Kase}
\address{Department of Pure and Applied Mathematics
Graduate School of Information Science and Technology
,Osaka University, Toyonaka, Osaka 560-0043, Japan}
\email{r-kase@cr.math.sci.osaka-u.ac.jp}
\date{}

\begin{abstract}
Happel and Unger defined a partial order on the set of basic tilting modules. The tilting quiver 
 is the Hasse diagram of the poset of basic tilting modules. We determine the 
 number of arrows in the tilting quiver over a path algebra of type $A$ or $D$.  
\end{abstract}
\maketitle 

\tableofcontents

\section*{Introduction}
 In this paper we use the following notations. Let $A$ be a
 finite dimensional algebra over an algebraically closed
 field $k$, and let mod-$A$ be the category of finite 
 dimensional right $A$-modules. For $M\in $mod-$A$ we denote by pd$_{A}M$ the projective
 dimension of $M$, and by $\add M$ the full subcategory
 of direct sums of direct summands of $M$. Let $Q=(Q_{0},Q_{1})$ be a finite connected quiver without loops and cycles, 
 and $Q_{0}$ (resp.$Q_{1}$) be the set of vertices (resp.arrows) of $Q$ (we use this notation for an arbitrary quiver).  We denote
 by $kQ$ the path algebra of $Q$ over $k$, and by $\rep Q$ the category of
 finite dimensional representations of the quiver $Q$ which is category equivalent to mod-$kQ$.
 For $M\in \rep Q$, denote by $M_{a}$
 the vector space of $M$ associated to a vertex $a$, and denote by $M_{a\rightarrow b}$ the
 linear map $M_{a}\rightarrow M_{b}$ of $M$.
 For a vertex $a$ of $Q$, let $\sigma _{a}Q$ be the quiver obtained from $Q$ by
 reversing all arrows starting at $a$ or ending at $a$. A module $T\in \modu A$ is called a tilting
  module provided the following three conditions are satisfied:\\
  (a) $\pd T<\infty$,\\
  (b) $\Ext ^{i}(T,T)=0$ for all $i>0$,\\
  (c) there exists an exact an sequence
   \[0\longrightarrow A\longrightarrow T_{0}\longrightarrow T_{1}\longrightarrow \cdots \longrightarrow T_{r}
\longrightarrow 0\ (T_{i}\in \add T)\]
in mod-$A$. In the hereditary case the tilting condition above is equivalent to the following:\\
  (a) $\Ext ^{1}(T,T)=0$,\\
  (b) the number of indecomposable direct summands of $T$ (up to isomorphism) is equal to the 
   number of simple modules.
   
    In section1,
 following \cite{HU1},\cite{HU2},\cite{HU3},
\cite{U}, 
we define a partial order on the set $Tilt(A)$ of all basic tilting modules (up to isomorphism)
 over $A$ and define the quiver of tilting modules $\vec{\mathcal{K}}(A)$.
 In Section2, we explain results from \cite{L}.  In Section3, we first show that the number of
 arrows of $\vec{\mathcal{K}}(kQ)$ is equal to the number of arrows of $\vec{\mathcal{K}}(kQ^{'})$ 
 if $Q$ and $Q^{'}$ share the same underlying graph 
 by applying the results from Section2. 
  Then  we determine the number of arrows 
 of $\vec{\mathcal{K}}(kQ)$ for any Dynkin quiver Q of type $A$ or $D$. Note that the underlying graph of
 $\vec{\mathcal{K}}(kQ)$ may be embeded into the exchange graph, or the cluster
 complex, of the corresponding cluster algebra of finite type:the tilting modules of $kQ$
 correspond to positive clusters \cite{BMRRT} and \cite{MRZ}.
 The number of positive clusters when the orientation is alternating is given in \cite[prop. 3.9]{FZ}.
 However, according to experts, the number of edges of this subdiagram of positive clusters is not
 known in the cluster tilting theory. Note also that if we consider
 the similar problem for the exchange graph, it is not 
interesting, because the number of edges is $\frac{n}{2}\times (\text{the number of vertices} )$, and the 
 number of vertices is given in \cite[prop. 3.8]{FZ}.
 The following is known.\cite[prop. 3.9]{FZ}. \[\# \vec{ \mathcal{K}}(kQ)_{0}=\left\{\begin{array}{ll}
\frac{1}{n+1}{2n\choose n} &\text{if}\ Q\ \text{is a Dynkin quiver of type}\ A_{n}, \\\\
\frac{3n-4}{2n}{2(n-1)\choose n-1} &\text{if}\ Q\ \text{is a Dynkin quiver of type}\ D_{n}. \\
\end{array}\right.\]

The main result of this paper is as follows.
\begin{thm} $(1):$Let $Q$ be a quiver without loops and cycles.
 Then $\#\vec{\mathcal{K}}(kQ)_{1}$ is independent of the orientation.\\
  $(2):$\[\# \vec{ \mathcal{K}}(kQ)_{1}=\left\{\begin{array}{ll}
{2n-1\choose n+1} &\text{if}\ Q\ \text{is a Dynkin quiver of type}\ A_{n}, \\\\
(3n-4){2(n-2)\choose n-3} &\text{if}\ Q\ \text{is a Dynkin quiver of type}\ D_{n}. \\
\end{array}\right.\]
\end{thm}
As a corollary of our theorem, we also have the following.
The exchange graph here is  defined in a purely combinatorial manner
 by using the notion of compatibility degree.
\begin{thm} We consider the root system $\Phi$ (resp.the positive root system $\Phi_{>0}$)
 of type $A_{n}$ or $D_{n}$. Let
 $E(\Phi )$ be the exchange graph of $\Phi $ $($see\ \cite[def. 1.14]{FZ}$)$ and $E(\Phi_{>0} )$
 the subgraph of $E(\Phi )$ whose vertices are positive clusters.
 Then the number of edges in $E(\Phi_{>0} )$ is
\[\left\{\begin{array}{ll}
{2n-1\choose n+1} &\text{if}\ \Phi \ \text{is of type}\ A_{n}, \\\\
(3n-4){2(n-2)\choose n-3} &\text{if}\ \Phi \ \text{is of type}\ D_{n}. \\
\end{array}\right.\]
\end{thm}
For the proof, let $Q$ be the alternating Dynkin quiver associated to $\Phi $.
 Then $($by \cite[cor. 4.12]{MRZ}$)$ 
 $E(\Phi_{>0} )$ coincides with the underlying graph of $\vec{\mathcal{K}}(kQ)$.

\section{Preliminaries}

In this section we define a partial order on tilting modules. First, for a tilting module $T$, we define the right 
 perpendicular category\[T^{\perp }=\{X\in \text{mod}\text{-} A\mid \mathrm{Ext}_{A}^{>0}(T,X)=0\}.\]
\begin{lem}$($\text{cf}.\cite[lemma2.1 (a)]{HU2}$)$ For tilting modules $T$, $T^{'}$ the following conditions are equivalent:\\
$\mathrm{(1)}:T^{\perp }\subset T^{'\perp }$, \\
$\mathrm{(2)}:T\in T^{'\perp }.$
\end{lem}
Recall that $Tilt(A)$ is the set of basic tilting modules of $A$.
\begin{defn} We define a partial order on $Tilt(A)$ by
\[T\leq T^{'}\stackrel{def}{\Longleftrightarrow }T^{\perp }\subset T^{'\perp }\Longleftrightarrow T\in T^{'\perp },\] 
for $T,T^{'}\in Tilt(A)$.
\end{defn}
 By definition, $A_{A}$ is the unique maximal element of $(Tilt(A),\leq )$. On the other hand, $(Tilt(A),\leq )$
 does not always admit a minimal element.\\
\begin{defn}(cf.\cite{U}) Let $\mathcal{C}$ be a full subcategory of mod-$A$ which is closed under direct sums, direct summands and 
isomorphisms. The subcategory $\mathcal{C}$ is called $contravariantly\ finite$ in $\modu A$, if every $X\in \modu A$ has a right
$\mathcal{C}$-approximation, i.e.there is a morphism $F_{X}\rightarrow X$ with $F_{X}\in \mathcal{C}$
 s.t.the induced morphism
 Hom$_{A}(C,F_{X})\rightarrow \mathrm{Hom}_{A}(C,X)$ is surjective for all $C\in \mathcal{C}$.
\end{defn}
\begin{thm}$($cf.\cite[thm 3.3 and cor 3.4]{HU1}$)$ $(Tilt(A),\leq )$ has a minimal element
 if and only if $\mathcal{P}^{< \infty}$ $($the full subcategory
 of $\modu A$ of modules with finite projective dimension$)$ is contravariantly finite in $\modu A$.
 Moreover the minimal element
 is unique if it exists.
\end{thm}
 Next we define the $tilting\ quiver$ $\vec{\mathcal{K}}(A)$, and recall its some properties. Let $\ind A$ be a category of indecomposable modules in $\modu A$. 
\begin{defn} The $tilting\ quiver$ $\vec{\mathcal{K}}(A)=(\vec{\mathcal{K}}(A)_{0},\vec{\mathcal{K}}(A)_{1})$ is defined
as follows.\\
(1)$\vec{\mathcal{K}}(A)_{0}=Tilt(A),$ \\
(2)$T^{'}\rightarrow T$ in $\vec{\mathcal{K}}(A)$, for $T,T^{'}\in Tilt(A)$, if $T^{'}=M\oplus X,\ T=M\oplus Y$
with $X,Y\in \ind A$ and there is a non-split short exact sequence
\[0\longrightarrow X\longrightarrow \widetilde{M} \longrightarrow Y\longrightarrow 0\]
with $\widetilde{M} \in \add M$.
\end{defn}
\begin{thm}$($cf.\cite[thm 2.1]{HU1}$)$ $\vec{\mathcal{K}}(A)$ is the $Hasse$-$diagram$ of $(Tilt(A),\leq )$
$($i.e.if \ $T\rightarrow T^{'}\in \vec{\mathcal{K}}(A)_{1}$ and $T\geq T^{''}\geq T^{'}$ then $T^{''}=T$ 
 or $T^{''}=T^{'})$.
\end{thm}
\begin{prop}$($cf.\cite[cor 2.2]{HU1}$)$ If $\vec{\mathcal{K}}(A)$ has a finite component $\mathcal{C}$,
 then $\vec{\mathcal{K}}(A)=\mathcal{C}$.
\end{prop}
\begin{cor} $\vec{\mathcal{K}}(A)$ contains at most one sink. Moreover it contains a sink if and only if 
 $\mathcal{P}^{<\infty }(A)$ is contravariantly finite in $\modu A$.
\end{cor}
\begin{prop}$($cf.\cite[cor 3.6]{HU1}$)$  Let $T\in Tilt(A)$ s.t.End$_{A}(T)$ is representation finite.
 Then $\mathcal{P}^{<\infty }(A)$ is
 contravariantly finite.
\end{prop}
\ \ Let $Q=(Q_{0},Q_{1})$ be a quiver without loops and cycles and $A=kQ$. For $T\in Tilt(A)$, let
\[\begin{array}{lll}
s(T) & = & \#\{T^{'}\in Tilt(A)\mid T\rightarrow T^{'}\ \mathrm{in}\ \vec{\mathcal{K}(kQ)}\} \\
e(T) & = & \#\{T^{'}\in Tilt(A)\mid T^{'}\rightarrow T\ \mathrm{in}\ \vec{\mathcal{K}(kQ)}\} \\
\end{array}\]
 and define $\delta (T)=s(T)+e(T)$.
\begin{prop}
 \label{HU prop}
 $($cf.\cite[prop 3.2]{HU3}$) \delta (T)=n-\#\{a\in Q_{0}\mid (\dimvec T)_{a}=1\}$, where $n=\#Q_{0}$.
\end{prop}

\section{A theorem of Ladkani}

In this section, we review \cite{L}. Let $Q$ be a quiver without loops and cycles and let $x$ be a source of $Q$.
Let $Tilt(Q):=Tilt(kQ)$ and define 
\[Tilt(Q)^{x}:=\{T\in Tilt(Q)\mid S(x)\mid T\},\]
where $S(x)$ is the simple module associated to $x$.
\begin{defn} Let $(X,\leq _{X})$,$(Y,\leq _{Y})$ be posets and $f:X\rightarrow Y$ an order-preserving function.
 Then we define 
 the preorder $\leq _{+}^{f}$, $\leq _{-}^{f}$ of $X\sqcup Y$ as follows.
\[a\leq _{+}^{f} b\Longleftrightarrow \left\{\begin{array}{ll}
a\leq _{X} b & \text{if}\ a,b\in X, \\
a\leq _{Y} b & \text{if}\ a,b\in Y, \\
f(a)\leq _{Y} b & \text{if}\ a\in X\ \text{and}\ b\in Y. \\
\end{array} \right.\]
\[a\leq _{-}^{f} b\Longleftrightarrow \left\{\begin{array}{ll}
a\leq _{X} b & \text{if}\ a,b\in X, \\
a\leq _{Y} b & \text{if}\ a,b\in Y, \\
a\leq _{Y} f(b) & \mathrm{if}\ a\in Y\ \mathrm{and}\ b\in X. \\
\end{array} \right.\] 
\end{defn} 
\begin{lem} Define the functors 
\[j^{-1}:\rep Q\longrightarrow \rep (Q\setminus \{x\})\] and 
\[j_{\ast }:\rep (Q\setminus \{x\})\longrightarrow \rep Q,\]
by
\[(j^{-1}M)_{a}=M_{a}\ \ \ (j^{-1}M)_{a\rightarrow b}=M_{a\rightarrow b} \] and 
\[(j_{\ast }N)_{a}=\left\{ \begin{array}{ll}
N_{a} & (a\neq x) \\
\oplus_{x\rightarrow y}N(y) & (a=x) \\
\end{array} \right. ,\ \ \ (j_{\ast }N)_{a\rightarrow b}=\left\{ \begin{array}{ll}
N_{a\rightarrow b} & (a\neq x) \\
(j_{\ast }N)_{x}\stackrel{projection}{\longrightarrow }N_{b} & (a=x) \\
\end{array} \right. .\]
Then $j^{-1}$ and $j_{\ast }$ are exact and $j_{\ast }$ is right adjoint to $j^{-1}$.
\end{lem}
 Denote by $\mathcal{D}^{b}(Q)$ the bounded derived category $\mathcal{D}^{b}(\rep Q)$.
\begin{lem} The functors $j^{-1}$ and $j_{\ast }$ induce functors
\[j^{-1}:\mathcal{D}^{b}(Q)\longrightarrow \mathcal{D}^{b}(Q\setminus \{x\}),\ \ \ j_{\ast }:\mathcal{D}^{b}(Q\setminus \{x\})\longrightarrow \mathcal{D}^{b}(Q)\]
with
\[\mathrm{Hom}_{\mathcal{D}^{b}(Q\setminus \{x\})}(j^{-1}M,N)\simeq \mathrm{Hom}_{\mathcal{D}^{b}(Q)}(M,j_{\ast }N),\]
for all $M\in \mathcal{D}^{b}(Q),\  N\in \mathcal{D}^{b}(Q\setminus \{x\})$.
\end{lem}
\begin{lem} The functors $j^{-1}$ and $j_{\ast }$ identify $\rep (Q\setminus \{x\})$
 with the right perpendicular subcategory
\[S(x)^{\perp }=\{M\in \rep Q\mid \mathrm{Ext}^{i}(S(x),M)=0\ for\ all\ i\geq0\}\]
of $\rep Q.$
\end{lem}
\begin{lem} The functor $j_{\ast }$ takes indecomposables of $\rep (Q\setminus \{x\})$
 to indecomposables of $\rep Q$.
\end{lem}
\begin{prop} Let $T$ be a tilting module in $\rep Q$. Then $j^{-1}T$ is a tilting module 
in $\rep (Q\setminus \{x\})$.\\\\
\ \ For $M=\oplus_{i=1}^{m} N_{i}^{r_{i}}$ (where $N_{i}\in \ind Q,r_{i}>0)$,
 let basic$(M)=\oplus_{i=1}^{m} N_{i}$.
\end{prop}
\begin{cor} The  map $\pi _{x}:T\mapsto \mathrm{basic}(j^{-1}T)$
 is an $order$-$preserving\ function$ 
\[(Tilt(Q),\leq)\rightarrow (Tilt(Q\setminus \{x\}),\leq ).\]
\end{cor}
\begin{prop} Let $T\in Tilt(Q\setminus \{x\}).$ Then $S(x)\oplus j_{\ast }T\in Tilt(Q)$.
\end{prop}
\begin{cor} The map $\iota_{x} :T\mapsto S(x)\oplus j_{\ast }T$
 is an $order$-$preserving\ function$
\[(Tilt(Q\setminus \{x\}),\leq )\rightarrow (Tilt(Q),\leq ).\]
\end{cor}
\begin{prop}
 \label{Ladkani prop}
  We have 
\[\pi_{x}\iota_{x}(T)=T,\]
for all $T\in Tilt(Q\setminus \{x\}).$ In addition, 
\[T\geq \iota_{x}\pi_{x}(T),\]
for all $T\in Tilt(Q)$, with equality if and only if $T\in Tilt(Q)^{x}$.

In particular, $\pi_{x}$ and $\iota_{x}$ induce
 an isomorphism of posets between $Tilt(Q)^{x}$ and $Tilt(Q\setminus \{x\})$.
\end{prop}
\begin{cor}
 \label{Ladkani cor}
  Let $X=Tilt(Q)\setminus Tilt(Q)^{x}$ and $Y=Tilt(Q)^{x}$. Define $f:X\rightarrow Y$
 by $f=\iota_{x} \pi_{x}$. Then 
\[Tilt(Q)\simeq (X\sqcup Y,\leq_{-}^{f}).\]
\end{cor}
 Now let $Q^{'}=\sigma_{x}Q$.
 Then $x$ is a sink of $Q^{'}$ and, by arguing in the similar way, we obtain the dual results by replacing 
\[(j^{-1},j_{\ast }, \pi_{x},\iota_{x},X,Y,f, \leq_{-}^{f})\]
 with 
\[(i^{-1},i_{!},\pi_{x}^{'},\iota_{x}^{'},X^{'},Y^{'},f^{'},\leq_{+}^{f^{'}}).\]  
 In particular we get
\[Tilt(Q^{'})^{x}\simeq Tilt(Q\setminus \{ x\}),\]
and
\[Tilt(Q^{'})\simeq (X^{'}\sqcup Y^{'},\leq_{+}^{f^{'}}),\]
where $X^{'}=Tilt(Q^{'})\setminus Tilt(Q^{'})^{x}$ and $Y^{'}=Tilt(Q^{'})^{x}$. 
\begin{thm}
 \label{Ladkani thm} 
 There exists an isomorphism of posets
 \[\rho:Tilt(Q)\setminus Tilt(Q)^{x}\rightarrow Tilt(Q^{'})\setminus Tilt(Q^{'})^{x}\] 
such that the following diagram commutes.\\\\
\ \ \ \ \ \ 
\unitlength 0.1in
\begin{picture}( 46.7000, 12.5000)( 12.0000,-15.6000)
\put(16.0000,-6.0000){\makebox(0,0)[lb]{$Tilt(Q)\setminus Tilt(Q)^{x}$}}%
%
\special{pn 8}%
\special{pa 2180 670}%
\special{pa 1450 1370}%
\special{fp}%
\special{sh 1}%
\special{pa 1450 1370}%
\special{pa 1512 1338}%
\special{pa 1488 1334}%
\special{pa 1484 1310}%
\special{pa 1450 1370}%
\special{fp}%
%
\special{pn 8}%
\special{pa 2880 530}%
\special{pa 4160 530}%
\special{fp}%
\special{sh 1}%
\special{pa 4160 530}%
\special{pa 4094 510}%
\special{pa 4108 530}%
\special{pa 4094 550}%
\special{pa 4160 530}%
\special{fp}%
\put(42.3000,-6.1000){\makebox(0,0)[lb]{$Tilt(Q^{'})\setminus Tilt(Q^{'})^{x}$}}%
\put(34.7000,-4.8000){\makebox(0,0)[lb]{$\rho_{x}$}}%
%
\special{pn 8}%
\special{pa 2620 660}%
\special{pa 3360 1400}%
\special{fp}%
\special{sh 1}%
\special{pa 3360 1400}%
\special{pa 3328 1340}%
\special{pa 3322 1362}%
\special{pa 3300 1368}%
\special{pa 3360 1400}%
\special{fp}%
\put(16.5000,-9.8000){\makebox(0,0)[lb]{$f$}}%
\put(27.6000,-9.9000){\makebox(0,0)[lt]{$\pi_{x}$}}%
\put(12.0000,-14.1000){\makebox(0,0)[lt]{$Tilt(Q)^{x}$}}%
\put(31.9000,-14.2000){\makebox(0,0)[lt]{$Tilt(Q\setminus \{x\})$}}%
%
\special{pn 8}%
\special{pa 3030 1500}%
\special{pa 1880 1500}%
\special{fp}%
\special{sh 1}%
\special{pa 1880 1500}%
\special{pa 1948 1520}%
\special{pa 1934 1500}%
\special{pa 1948 1480}%
\special{pa 1880 1500}%
\special{fp}%
\put(23.7000,-14.3000){\makebox(0,0)[lb]{$\iota_{x}$}}%
\put(23.5000,-15.5000){\makebox(0,0)[lt]{$\sim $}}%
\put(34.3000,-5.9000){\makebox(0,0)[lt]{$\sim $}}%
%
\special{pn 8}%
\special{pa 4610 650}%
\special{pa 3900 1390}%
\special{fp}%
\special{sh 1}%
\special{pa 3900 1390}%
\special{pa 3962 1356}%
\special{pa 3938 1352}%
\special{pa 3932 1328}%
\special{pa 3900 1390}%
\special{fp}%
%
\special{pn 8}%
\special{pa 5120 660}%
\special{pa 5870 1400}%
\special{fp}%
\special{sh 1}%
\special{pa 5870 1400}%
\special{pa 5838 1340}%
\special{pa 5832 1364}%
\special{pa 5808 1368}%
\special{pa 5870 1400}%
\special{fp}%
%
\special{pn 8}%
\special{pa 4230 1500}%
\special{pa 5530 1500}%
\special{fp}%
\special{sh 1}%
\special{pa 5530 1500}%
\special{pa 5464 1480}%
\special{pa 5478 1500}%
\special{pa 5464 1520}%
\special{pa 5530 1500}%
\special{fp}%
\put(56.3000,-15.9000){\makebox(0,0)[lb]{$Tilt(Q^{'})^{x}$}}%
\put(47.3000,-15.6000){\makebox(0,0)[lt]{$\sim $}}%
\put(47.4000,-14.2000){\makebox(0,0)[lb]{$\iota_{x}^{'}$}}%
\put(43.1000,-9.8000){\makebox(0,0)[lt]{$\pi_{x}^{'}$}}%
\put(54.5000,-8.1000){\makebox(0,0)[lt]{$f^{'}$}}%
\end{picture}%

\end{thm}
\begin{cor} $\# Tilt(Q)=\# Tilt(Q^{'}).$
\end{cor}
\begin{rem}In \cite{L} the partial order on $Tilt(A)$ is defined by 
\[T\geq T^{'}\Longleftrightarrow T^{\perp }\subset T^{'\perp}\ \; (\mathrm{opposite\ to\ our\ definition}).\]
\end{rem}

\section{Main results}

In this section we determine the number of arrows 
 of $\vec{\mathcal{K}}(kQ)$ in the case $Q$ is a 
 Dynkin quiver of type $A$ or $D$. Let 
\[\begin{array}{lll}\
Gen (M) & := & \{N\in \modu A\mid M^{'}\stackrel{surjection}{\rightarrow }N\ \mathrm{for\ some}\ M^{'}\in \add M\}\\
\Cogen (M) & := & \{N\in \modu A\mid N\stackrel{injection}{\rightarrow }M^{'}\ \mathrm{for\ some}\ M^{'}\in \add M\}\\
\end{array}.\]
\begin{lem}
 \label{lem1}
 $($cf.\cite[prop 1.3]{CHU}$)$ Let $A$ be hereditary, $T=M\oplus Y\in Tilt(A)$ with $Y\in \ind A$. 
 If $Y\in \Gen (M)$, then there exists a unique(up to isomorphism) indecomposable module $X$
 which is not isomorphic to $Y$ s.t.$M\oplus X\in Tilt(A)$ and there exists an exact sequence
\[0\longrightarrow X\longrightarrow E\longrightarrow Y\longrightarrow 0\]with $E\in \add M$.\\\\
 Dually, if $Y\in \Cogen (M)$ then there exists a unique(up to isomorphism) indecomposable module $X$ 
 which is not isomorphic to $Y$ s.t. $M\oplus Y\in Tilt(A)$ and there exists an exact sequence
\[0\longrightarrow Y\longrightarrow E\longrightarrow X\longrightarrow 0\]with $E\in \add M$.
\end{lem}
\begin{lem}
 \label{lem2}
  Let $Q$ be a quiver without loops and cycles.  If $x$ is a sink, 
 then for all $T=M\oplus S(x)\in Tilt(Q),\ S(x)$ is in $\Cogen (M)$.
 If $x$ is a source, then for all $T=M\oplus S(x)\in Tilt(Q),\ S(x)$ is in $\Gen (M)$.
\end{lem}
\begin{proof} By Proposition\ref{Ladkani prop}, \[Tilt(Q\setminus \{x\})\stackrel {1:1}{\longleftrightarrow} Tilt(Q)^{x}: T\longmapsto F(T)\oplus S(x)\]

where $F(T)\in \modu kQ$ is defined by \[F(T)_{a}=\left\{\begin{array}{llll}
T_{a}&\mathrm{if}\ a\neq x, \\
\oplus_{y\rightarrow x}T_{y}&\mathrm{if}\ a=x\ \mathrm{and}\ x\ \mathrm{is\ a}\ \text{sink}, \\
\oplus_{x\rightarrow y}T_{y}&\mathrm{if}\ a=x\ \mathrm{and}\ x\ \mathrm{is\ a}\ \text{source}. \\
\end{array} \right. \]

\[F(T)_{a\rightarrow b}=\left\{\begin{array}{llll}
T_{a\rightarrow b}& \mathrm{if}\ a,b\neq x,\\
T_{y}\stackrel{injection}{\longrightarrow}\oplus_{y^{'}\rightarrow x} T_{y^{'}}& \mathrm{if}\ a=y\ \mathrm{with}\ y\rightarrow x\ \mathrm{and\ if}\ b=x\ \mathrm{and}\ x\ \mathrm{is\ a}\ \text{sink}, \\
\oplus_{x\rightarrow y^{'}} T_{y^{'}}\stackrel{projection}{\longrightarrow} T_{y}& \mathrm{if}\ b=y\ \mathrm{with}\ x\rightarrow y\ \mathrm{and\ if}\ a=x\ \mathrm{and}\ x\ \mathrm{is\ a}\ \text{source}. \\
\end{array} \right. \]
Now if $x$ is a sink then \[S(x)\in \Cogen (M)\Longleftrightarrow M_{x}\neq 0,\]
 and if $x$ is a source then \[S(x)\in \Gen (M)\Longleftrightarrow M_{x}\neq 0.\]

So the lemma  follows from the fact that if $T\in Tilt(Q)$ then $(\dimvec T)_{a}\geq 1$, for all $a$.
\end{proof}
\begin{lem}
 \label{lem3}
  If $x$ is a sink then \[\{\alpha\in \vec{\mathcal{K}}(Q)_{1}\mid s(\alpha )\in Tilt(Q)^{x},
 t(\alpha )\in Tilt(Q)\setminus Tilt(Q)^{x}\} \stackrel {1:1}{\longleftrightarrow} Tilt(Q)^{x}.\]
If $x$ is a source then \[\{\alpha\in \vec{\mathcal{K}}(Q)_{1}\mid t(\alpha )\in Tilt(Q)^{x},
 s(\alpha )\in Tilt(Q)\setminus Tilt(Q)^{x}\} \stackrel {1:1}{\longleftrightarrow} Tilt(Q)^{x}.\]
 Where, for $T\stackrel{\alpha }{\rightarrow} T^{'}$, $s(\alpha)=T$ and $t(\alpha)=T^{'}$.
\end{lem}
\begin{proof} Suppose $x$ is a sink, and let $T\in Tilt(Q)^{x}$.
 Then there exists a unique $T^{'}\in Tilt(Q)\setminus Tilt(Q)^{x}$
 s.t. $T\longrightarrow T^{'}$ in $\vec{\mathcal{K}}(Q)$ (by lemma\ref{lem1},\ref{lem2}).

On the other hand, let $T^{'}\in Tilt(Q)\setminus Tilt(Q)^{x}$ and suppose that there exists $T_{1}$, $T_{2}\in Tilt(Q)^{x}$ s.t.
 $T_{1}\longrightarrow T^{'}$, $T_{2}\longrightarrow T^{'}$, for $T^{'}\in Tilt(Q)\setminus Tilt(Q)^{x}$,
 in $\vec{\mathcal{K}}(Q)$. Write $T_{i}=M\oplus S(x)\oplus Y_{i}$ with $Y_{i}\in \ind kQ (i=1,2)$
 then $Y_{i}\mid T^{'}$; $\Ext (Y_{i},Y_{j})=0\ (i,j=1,2)$. Thus $\Ext (T_{1}\oplus Y_{2},T_{1}\oplus Y_{2})=0$ and
$Y_{1}=Y_{2}$ follows.
\end{proof}
\begin{cor}
 \label{cor4}
  \[\# \vec{\mathcal{K}}(Q)_{1}= \# \vec{\mathcal{K}}(\sigma_{x}Q)_{1}.\]
In particular, if $Q$ is a Dynkin quiver then $\# \vec{\mathcal{K}}(Q)_{1}$ depends only on the underlying graph of $Q$.
\end{cor}
\begin{proof} By corollary\ref{Ladkani cor} and lemma\ref{lem3},
\[\begin{array}{ccl}
\# \vec{\mathcal{K}}(Q)_{1}&=&\# \vec{\mathcal{K}}(Q\setminus \{x\})_{1}+\#\vec{\mathcal{K}}(Tilt(Q)\setminus Tilt(Q)^{x})_{1}+\# Tilt(Q)^{x} \\
 &=& \# \vec{\mathcal{K}}(\sigma_{x}Q)_{1}. \\
\end{array}\]
\end{proof}

\subsection{case $A$}

In this subsection we consider the quiver,
\[Q=\stackrel{1}{\circ }\rightarrow \stackrel{2}{\circ }\rightarrow\dots \rightarrow \stackrel{n}{\circ }.\]

By Gabriel's theorem, $\ind kQ=\{L(i,j)\mid 0\leq i<j\leq n\}$ where \[L(i,j)=\left\{ \begin{array}{ll}
k & (i<a\leq j), \\
0 & \text{otherwise}, \\
\end{array} \right. L(i,j)_{a\rightarrow b}=\left\{\begin{array}{ll}
1 & (i< a,b\leq j), \\
0 & \text{otherwise}. \\
\end{array} \right.\] 
And
\[\tau L(i,j)=\left\{ \begin{array}{ll}
L(i+1,j+1) & (j<n), \\
0 & (j=n), \\
\end{array}\right. \] where $\tau$ is a Auslander-Reiten translation. 
 
\begin{defn} A pair of intervals $([i,j],[i^{'},j^{'}])$ is $compatible$ if 
\[ [i,j]\cap [i^{'},j^{'}]=\emptyset \ \mathrm{or}\  [i,j]\subset [i^{'},j^{'}]\ \mathrm{or}\ [i^{'},j^{'}]\subset [i,j].\]
\end{defn}
Applying Auslander-Reiten duality,
\[\mathrm{DExt}(M,N)\cong \mathrm{Hom}(N,\tau M)\ (\mathrm{D}=\mathrm{Hom}_{k}(-,k)),\]
  we get the following lemma.
\begin{lem}
 \label{lem6}
  We have \[\Ext (L(i,j),L(i^{'},j^{'}))=0=\Ext (L(i^{'},j^{'}),L(i,j))\]
if and only if  $([i,j], [i^{'},j^{'}])$ is compatible.
\end{lem}
\begin{proof} It is obvious that Hom$(L(i,j),L(i^{'},j^{'}))\neq 0$ if and only if $i^{'}\leq i \leq j^{'}\leq j$.
 So the lemma follows from this fact and the $AR$-duality.
\end{proof}
\begin{lem}
 \label{lem7}
  For any $T\in Tilt(Q)$, we get \[ \delta (T)=n-1.\]
\end{lem}
\begin{proof} Let $T\in Tilt(Q)$ then the projective-injective module $L(0,n)$ is a direct summand of $T$.
From this fact,  we get $\delta (T)<n $. 

Denote by $X$ the set of indecomposable direct summands of $T$ not isomorphic to  $L(0,n)$ and define 
\[ a:=\left\{ \begin{array}{ll}
max\{ i\mid L(0,i)\in X\} & \text{if}\  L(0,i)\in X\ \text{for some}\ i, \\
0 & \text{otherwise.} \\
\end{array}\right.\]

Then, by lemma\ref{lem6}, we get 
\[\Ext (T,L(a+1,n))=0=\Ext (L(a+1,n),T).\]
\ \ By $\Ext =0$ condition, we can see $L(a+1,n)$ is a direct summand of $T$. In particulur   
\[(\dimvec T)_{i}=1\Longleftrightarrow i=a+1.\]
The lemma follows from this fact and proposition\ref{HU prop} .
\end{proof}
Now it is easy to check the number of arrows in $\vec{\mathcal{K}}(Q)$, because it is equal to
\[\frac{1}{2}\sum_{T\in \mathrm{Tilt(Q)}} \delta (T)\]
\begin{cor} $\#\vec{\mathcal{K}}(Q)_{1}=\frac{n-1}{2(n+1)}{2n\choose n}={2n-1\choose n-2}.$
\end{cor}

\subsection{case $D$}

Through this subsection, we consider the quiver\\\\
\unitlength 0.1in
\begin{picture}( 27.3500,  5.0000)( 10.2000, -7.7500)
%
\special{pn 8}%
\special{ar 1856 536 42 42  0.0000000 6.2831853}%
%
\special{pn 8}%
\special{ar 2146 536 40 40  0.0000000 6.2831853}%
%
\special{pn 8}%
\special{pa 1916 536}%
\special{pa 2086 536}%
\special{fp}%
\special{sh 1}%
\special{pa 2086 536}%
\special{pa 2018 516}%
\special{pa 2032 536}%
\special{pa 2018 556}%
\special{pa 2086 536}%
\special{fp}%
%
\special{pn 8}%
\special{pa 2206 536}%
\special{pa 2376 536}%
\special{fp}%
\special{sh 1}%
\special{pa 2376 536}%
\special{pa 2308 516}%
\special{pa 2322 536}%
\special{pa 2308 556}%
\special{pa 2376 536}%
\special{fp}%
%
\special{pn 8}%
\special{pa 3186 546}%
\special{pa 3396 546}%
\special{fp}%
\special{sh 1}%
\special{pa 3396 546}%
\special{pa 3328 526}%
\special{pa 3342 546}%
\special{pa 3328 566}%
\special{pa 3396 546}%
\special{fp}%
%
\special{pn 8}%
\special{ar 3506 546 40 40  0.0000000 6.2831853}%
%
\special{pn 8}%
\special{sh 1}%
\special{ar 2526 536 10 10 0  6.28318530717959E+0000}%
\special{sh 1}%
\special{ar 2726 536 10 10 0  6.28318530717959E+0000}%
\special{sh 1}%
\special{ar 2956 536 10 10 0  6.28318530717959E+0000}%
\special{sh 1}%
\special{ar 2536 536 10 10 0  6.28318530717959E+0000}%
\special{sh 1}%
\special{ar 2736 536 10 10 0  6.28318530717959E+0000}%
\special{sh 1}%
\special{ar 2966 536 10 10 0  6.28318530717959E+0000}%
\special{sh 1}%
\special{ar 2966 536 10 10 0  6.28318530717959E+0000}%
%
\special{pn 8}%
\special{pa 3556 516}%
\special{pa 3656 406}%
\special{fp}%
\special{sh 1}%
\special{pa 3656 406}%
\special{pa 3596 442}%
\special{pa 3620 444}%
\special{pa 3626 468}%
\special{pa 3656 406}%
\special{fp}%
%
\special{pn 8}%
\special{pa 3556 586}%
\special{pa 3656 686}%
\special{fp}%
\special{sh 1}%
\special{pa 3656 686}%
\special{pa 3622 624}%
\special{pa 3618 648}%
\special{pa 3594 652}%
\special{pa 3656 686}%
\special{fp}%
%
\special{pn 8}%
\special{ar 3716 376 40 40  0.0000000 6.2831853}%
%
\special{pn 8}%
\special{ar 3716 736 40 40  0.0000000 6.2831853}%
\put(18.5500,-4.1500){\makebox(0,0){$1$}}%
\put(21.5000,-4.2000){\makebox(0,0){$2$}}%
\put(34.5000,-4.2000){\makebox(0,0){$n-1$}}%
\put(39.7000,-3.6000){\makebox(0,0){$n^{+}$}}%
\put(39.7000,-7.2000){\makebox(0,0){$n^{-}$}}%
\put(10.2000,-6.1000){\makebox(0,0)[lb]{$Q=Q_{n}=$}}%
\end{picture}%
\\\\
\ \ Then $\ind kQ=\{L(a,b)\mid 0\leq a<b\leq n-1\}\\
\ \ \ \ \ \ \ \ \ \ \ \ \ \ \ \ \ \ \ \ \cup \{L^{\pm }(a,n)\mid 0\leq a\leq n-1\}\cup \{M(a,b)\mid 0\leq a<b\leq n-1\}$ \\
\ where \[\begin{array}{lll}
L(a,b)_{i} & = & \left\{ \begin{array}{cl}
k & \mathrm{if }\  a<i\leq b, \\
0 & \mathrm{otherwise,} \\
\end{array} \right. \\\\
  L(a,b)_{i\rightarrow j} & = & \left\{ \begin{array}{cl}
1 & \mathrm{if }\ a<i<b, \\
0 & \mathrm{otherwise,} \\
\end{array} \right. \\\\ 
L(a,n)^{\pm }_{i} & = & \left\{ \begin{array}{cl}
k & \mathrm{if }\ a<i\leq n-1\ or\ i=n^{\pm }, \\
0 & \mathrm{otherwise,} \\
\end{array} \right. \\\\
 L(a,n)^{\pm }_{i\rightarrow j} & = & \left\{ \begin{array}{cl}
1 & \mathrm{if }\ a<i<n-1\ or\ i=n-1,j=n^{\pm },  \\
0 & \mathrm{otherwise,} \\
\end{array} \right. \\\\
M(a,b)_{i} & = & \left\{ \begin{array}{ll}
k & \mathrm{if }\ a<i\leq b\ \mathrm{or}\ i=n^{\pm }, \\
k^{2} & \mathrm{if }\ b<i\leq n-1, \\
0 & \mathrm{otherwise,} \\
\end{array} \right. \\\\
 M(a,b)_{i\rightarrow j} & = & \left\{ \begin{array}{cl}
1 & \mathrm{if }\ a<i<b, \\
\left( \begin{array}{l}
1 \\
1 \\
\end{array} \right) & \mathrm{if }\ i=b, \\\\
(1,0) & \mathrm{if }\ i=n-1,j=n^{+}, \\\\
(0,1) & \mathrm{if }\ i=n-1,j=n^{-},\\\\
\left(\begin{array}{cc}
1 & 0 \\
0 & 1 \\
\end{array}\right) & \mathrm{if }\ b<i<n-1, \\\\
0 & \text{otherwise.} \\
\end{array} \right. \\
\end{array} \]
Then \\
$\tau L(a,b)=\left\{ \begin{array}{ll}
L(a+1,b+1) & \mathrm{if }\ b<n-1, \\
M(0,a+1) & \mathrm{if }\ b=n-1, \\
\end{array} \right. \\
\tau L^{+}(a,n)  = L^{-}(a+1,n),\\
\tau L^{-}(a,n)  = L^{+}(a+1,n),\\
\tau M(a,b)  =  \left\{ \begin{array}{ll}
M(a+1,b+1) & \mathrm{if }\ b<n-1, \\
0 & \mathrm{if }\ b=n-1. \\
\end{array} \right.$\\
\begin{lem}
 \label{lem9} \begin{align*}
(1) & \Ext (L(a,b),L(a^{'},b^{'}))=0=\Ext (L(a^{'},b^{'}),L(a,b)) \\
 & \Longleftrightarrow  ([a,b], [a^{'},b^{'}]) :\mathrm{compatible.} \\
(2) & \Ext (L(a,b),L^{\pm }(a^{'},n))=0=\Ext (L^{\pm }(a^{'},n),L(a,b)) \\
 &  \Longleftrightarrow  ([a,b], [a^{'},n]) :\mathrm{compatible.} \\
(3) & \Ext(L(a,b),M(a^{'},b^{'}))=0=\Ext(M(a^{'},b^{'}),L(a,b)) \\
 &  \Longleftrightarrow  ([a,b], [a^{'},n]),([a,b],[b^{'},n]) :\mathrm{compatible.} \\
(4) & \Ext (M(a,b),L^{\pm }(a^{'},n))=0=\Ext (L^{\pm }(a^{'},n),M(a,b)) \\
 &  \Longleftrightarrow  a\leq a^{'} \leq b. \\
(5) & \mathrm{Ext}(L^{\pm }(a,n),L^{\pm }(a^{'},n))=0=\Ext (L^{\pm }(a^{'},n),L^{\pm }(a,n))\ \ \mathrm{for\ all\ } a,a^{'}. \\
(6) & \Ext (L^{+}(a,n),L^{-}(a^{'},n))=0=\Ext (L^{-}(a^{'},n),L^{+}(a,n))  \\
 &  \Longleftrightarrow  a=a^{'}. \\
(7) & \Ext (M(a,b),M(a^{'},b^{'}))=0=\Ext (M(a^{'},b^{'}),M(a,b)) \\
 &  \Longleftrightarrow  [a,b]\subset [a^{'},b^{'}]\ or\ [a^{'},b^{'}]\subset [a,b]. \\
\end{align*}
\end{lem}
\begin{proof} (1) and (2) follow from the case $A$ and (5),(6) are obvious.

(3):(case $b<a^{'}$) It is obvious that 
\[\Ext (L(a,b),M(a^{'},b^{'}))=0=\Ext (M(a^{'},b^{'}),L(a,b)).\]
(case $a<a^{'}\leq b<b^{'}$) In this case we claim that
\[\Hom (M(a^{'},b^{'}),\tau L(a,b))\neq 0.\]
In fact $0\neq f=(f_{i})_{i}\in \Hom (M(a^{'},b^{'}),\tau L(a,b))$ 
 where \[f_{i}=\left\{ \begin{array}{cl}
1 & \mathrm{if }\ a^{'}<i\leq b+1, \\
0 & \mathrm{otherwise}. \\
\end{array} \right.\]
(case $a<a^{'}<b^{'}\leq b<n-1$) In this case we claim that
\[\Hom (M(a^{'},b^{'}),\tau L(a,b))\neq 0.\]
In fact $0\neq f=(f_{i})_{i}\in \Hom (M(a^{'},b^{'}),\tau L(a,b))$ 
 where \[f_{i}=\left\{ \begin{array}{cl}
1 & \mathrm{if }\ a^{'}<i\leq b^{'}, \\
(0,1) & \mathrm{if }\ b^{'}<i\leq b, \\
0 & \mathrm{otherwise}. \\
\end{array} \right.\]
(case $a<a^{'}<b^{'}\leq b=n-1$) In this case we also claim that
\[\Hom (M(a^{'},b^{'}),\tau L(a,b))\neq 0.\]
In fact $0\neq f=(f_{i})_{i}\in \Hom (M(a^{'},b^{'}),\tau L(a,n-1))$ 
 where \[f_{i}=\left\{ \begin{array}{cl}
\left( \begin{array}{l}
1 \\
1 \\
\end{array} \right) & \mathrm{if }\ a^{'}<i\leq b^{'}, \\
1 & \mathrm{if }\ b^{'}<i\leq n-1\ \mathrm{or}\ i=n^{\pm }, \\
0 & \mathrm{otherwise}. \\
\end{array} \right.\]
(case $a^{'}\leq a<b<b^{'}<n-1$) In this case we claim that
\[\Hom (M(a^{'},b^{'}),\tau L(a,b))=0=\Hom (L(a,b),\tau M(a^{'},b^{'})).\]
Let $f=(f_{i})_{i}\in \Hom (M(a^{'},b^{'}),\tau L(a,b))$.
 Then $f_{i}=0$ if $i\leq a+1$ or $b+1<i\leq n-1$ or $i=n^{\pm }$ and 
\[f_{a+2}=f_{a+3}=\cdots =f_{b+1}.\]
Now the commutative square for $f_{a+1},f_{a+2}$ shows $f_{a+2}=0$.
 So \[\Hom (M(a^{'},b^{'}),\tau L(a,b))=0.\]
And similarly \[\Hom (L(a,b),\tau M(a^{'},b^{'}))=0.\]
(case $a^{'}\leq a<b<b^{'}=n-1$) It is obvious that  
\[\Hom (M(a^{'},n-1),\tau L(a,b))=0\]
and, since $M(a^{'},n-1)$ is projective, we have \[\Hom (L(a,b),\tau M(a^{'},b^{'}))=0.\]
(case $a^{'}\leq a<b^{'}\leq b<n-1$) In this case we claim that
\[\Hom (M(a^{'},b^{'}),\tau L(a,b))\neq 0.\]
In fact $0\neq f=(f_{i})_{i}\in \Hom (M(a^{'},b^{'}),\tau L(a,b))$ 
 where \[f_{i}=\left\{ \begin{array}{cl}
(1,-1) & \mathrm{if }\ b^{'}<i\leq b+1, \\
0 & \mathrm{otherwise}. \\
\end{array} \right.\]
(case $a^{'}\leq a<b^{'}\leq b=n-1$) In this case we also claim that
\[\Hom (M(a^{'},b^{'}),\tau L(a,b))\neq 0.\]
In fact $0\neq f=(f_{i})_{i}\in \Hom (M(a^{'},b^{'}),\tau L(a,n-1))$ 
 where \[f_{i}=\left\{ \begin{array}{cl}
1 & \mathrm{if }\ a^{'}<i\leq a+1\ \text{or}\ i=n^{\pm },   \\\\
\left( \begin{array}{l}
1 \\
1 \\
\end{array} \right) & \mathrm{if }\ a+1<i\leq b^{'}, \\\\
\left(\begin{array}{cc}
1 & 0 \\
0 & 1 \\
\end{array}\right) & \mathrm{if }\ b^{'}<i\leq n-1, \\\\
0 & \mathrm{otherwise}. \\
\end{array} \right.\]
(case $b^{'}\leq a$) Similar to the case $a^{'}\leq a<b<b^{'},$ we get
\[\Hom (M(a^{'},b^{'}),\tau L(a,b))=0=\Hom (L(a,b),\tau M(a^{'},b^{'})).\]
So we have proved (3).

(4),(7):The proofs are similar to (3).
\end{proof}
\begin{lem} 
 \label{510}
  Let $T\in Tilt(Q)$. 

$(1) L(0,n-1)\mid T$ implies $L^{\pm}(0,n)\mid T$.

$(2)$ If $L^{+}(0,n)\mid T\ (resp.L^{-}(0,n)\mid T)$ and  all 
 indecomposable direct summands of $T$ are insincere, then
 \[ L^{-}(0,n)\mid T\ (resp.L^{+}(0,n)\mid T).\]
\end{lem}
\begin{proof} $(1):$Suppose $L(0,n-1)\mid T$. Then 
\[\Ext (T,L(0,n-1))=0=\Ext (L(0,n-1),T)\] and there exists an injection 
\[\tau L^{\pm }(0,n)\longrightarrow \tau L(0,n-1).\]
So we get \[\Ext (L^{\pm } (0,n),T)\simeq \Hom (T,\tau L^{\pm }(0,n))=0.\]
Since $L^{\pm }(0,n)$ is injective, we also get 
\[\Ext (T,L^{\pm }(0,n))=0.\] 
Therefore, $L^{\pm }(0,n)\mid T$.\\\\
\ \ $(2)$: Suppose $L^{+}(0,n)\mid T$ and that all 
indecomposable direct summands of $T$ are insincere.
 Now $(\dimvec T)_{n^{-}}\neq 0,$ so there exists some indecomposable direct summand
 $N$ s.t.\[(\dimvec N)_{n^{-}}\neq 0.\]
If $N=M(a,b)$ then $\Ext (M(a,b),L^{+}(0,n))=0=\Ext (L^{+}(0,n),M(a,b))$ so $a=0$ and $N$ is sincere. This is a contradiction.
 So $N=L^{-}(a,n)$ and $a=0$ by $L^{+ }(0,n)\mid T$.
\end{proof}
\begin{lem}
 \label{511}
  For all $T\in Tilt(Q)$ there exists some indecomposable direct summand $N$ of $T$ s.t.
\[(\dimvec N)_{i}\geq 1,\ \mathrm{for\ all}\ i\leq n-1.\] 
Thus, $N=L(0,n-1),L^{\pm }(0,n)$ or $M(0,b)$, for some $b$.
\end{lem}
\begin{proof} For an indecomposable direct summand $N$ of $T$ s.t.$(\dimvec N)_{1}=1$,
 define \[a(N)\stackrel {def}{=} \sup \{i\mid 1\leq i\leq n-1, (\dimvec N)_{i}\geq 1\}.\]
Suppose that $\sup\ a(N)=a<n-1,$ then $L(0,a)\mid T$.
 So indecomposable direct summands of $T$ are of the following form
\[\begin{array}{ll}
L(a^{'},b^{'}) & \mathrm{for }\ b^{'}\leq a\ \mathrm{or}\ a+1\leq a^{'}, \\
L^{+}(a^{'},n) & \mathrm{for }\ a+1\leq a^{'}, \\
M(a^{'},b^{'}) & \mathrm{for }\ a+1\leq a^{'}. \\
\end{array} \]
So $(\dimvec T)_{a+1}=0.$ This is a contradiction.\end{proof}
\begin{lem}
 \label{512}
 We have \[\#\{i\mid 1\leq i\leq n-1,\ (\dimvec T)_{i}=1\}\leq 1.\]
In particular, $\delta (T)\geq n-2$. 
\end{lem}
\begin{proof} Let $i\neq n^{\pm }$ s.t. $(\dimvec T)_{i}=1.$ Then we claim that
\[L(0,i-1)\mid T.\]
By lemma\ref{511} there exists a unique indecomposable direct summand $N$ of $T$ s.t.
\[(\dimvec N)_{j}\geq 1\ \mathrm{for\ all}\ j\leq n-1.\]
So, by lemma\ref{510}, $N=M(0,j)$ for some $i \leq j\leq n-1$ and any indecomposable direct summand of $T$ not isomorphic to $N$ is one of the following,
\[\begin{array}{ll}
L(a,b) & \mathrm{for }\ b\leq i-1\ \text{or}\ i\leq a, \\
L^{\pm }(a,n) & \mathrm{for }\ i\leq a, \\
M(a,b) & \mathrm{for }\ i\leq a. \\
\end{array}\]
It implies  \[\Ext (T,L(0,i-1))=0=\Ext(L(0,i-1),T),\] so that \[L(0,i-1)\mid T.\]
\end{proof}
\begin{cor}
 \label{513}
  Let $T\in Tilt(Q)$ then $\delta (T)\geq n-1$,
 and $\delta (T)=n-1$ if and only if $L^{\pm }(0,n)\mid T$
 and other indecomposable direct
 summands of $T$ have the form $L(a,b) (0\leq a<b\leq n-1)$. In particular, 
\[\#\{T\in Tilt(kQ)\mid \delta (T)=n-1\}=\frac{1}{n}{2(n-1)\choose n-1}=\frac{1}{n-1}{2(n-1)\choose n-2}.\]
\end{cor}
\begin{proof} Suppose that all indecomposable direct summands of $T$ are insincere.
Then, by lemma\ref{510} and \ref{511}, $L^{+}(0,n)$ and $L^{-}(0,n)$ are both direct summands of $T$.
 So $(\dimvec T)_{i}=1$ if and only if $i=n^{\pm }$. So we have $\delta (T)\geq n-1$.
 If the equality holds
 then indecomposable direct summands of $T$ not isomorphic to $L^{\pm }(0,n)$ are of the form $L(a,b)$.

Next we suppose there is a sincere indecomposable direct summand $N$ of $T$. If $\delta (T)=n-2$ then, by lemma\ref{512}, there is a unique $i\leq n-1$ s.t.
\[(\dimvec T)_{i}=(\dimvec T)_{n^{\pm }}=1.\]
 So all indecomposable direct summands
 of $T$ not isomorphic to $N$ are of the form $L(a,b)\ (b<i$ or $i\leq a).$ As their direct sum may be viewed as a rigid module
 in type $A_{i-1}\times A_{n-i-1},$ we get  
\[\# \{L(a,b)\mid \ L(a,b)\mid T\}\leq (i-1)+(n-1-i)=n-2,\] which is a contradiction.
Next we consider the case  $\delta (T)=n-1$.

$(a):(\dimvec T)_{i}=(\dimvec T)_{n^{+}}=1$, for a unique $i(\leq n-1)$.
 Then indecomposable direct summands of $T$ not isomorphic to $N$ are of the following form:      
\[\begin{array}{ll}
L(a,b) & \mathrm{for }\ b< i\ \mathrm{or}\ i\leq a, \\
L^{-}(a,n) & \mathrm{for }\ i\leq a .\\
\end{array}\]
We get by the same argument that
\[\# \{L\in \ind kQ\mid \ L\mid T,\ L\neq N\}\leq (i-1)+(n-i)=n-1,\]
 which is a contradiction.

$(b):(\dimvec T)_{i}=(\dimvec T)_{n^{-}}=1$, for a unique $i(\leq n-1)$. Then, similar to $(a)$,
 we reach a contradiction.

$(c):(\dimvec T)_{n^{\pm }}=1.$
 Then indecomposable direct summands of $T$ not isomorphic to $N$ are of the form $L(a,b)$. Thus
 \[\# \{L(a,b)\mid \ L(a,b)\mid T\}\leq n-1.\]
 It is a contradiction. So we get $\delta (T)\geq n$ and $\delta (T)=n-1$ does not occur in this case.
 
Thus we have proved that if $\delta (T)=n-1$ then $L^{\pm }(0,n)\mid T$ and the other indecomposable 
 direct summands of $T$ has the form $L(a,b)$. The converse implication is clear.
\end{proof}
Now we define subsets of $Tilt(Q)$ by 
\[\begin{array}{lll}
\mathcal{T}_{0} & := & \{T\in Tilt(Q)\mid \delta (T)=n+1\}, \\
\mathcal{T}_{1} & := & \{T\in Tilt(Q)\mid \delta (T)=n\}, \\
\mathcal{T}_{2} & := & \{T\in Tilt(Q)\mid \delta (T)=n-1\}. \\
\end{array}\]
\begin{lem}
 \label{514}
  Fix $1\leq i\leq n-1$, then 
\[\begin{array}{l}
\{T\in \mathcal{T}_{1} \mid (\dimvec T)_{i}=1\} \stackrel {1:1}{\longleftrightarrow} \\
 Tilt(\circ \rightarrow \circ \rightarrow
 \cdots \rightarrow \stackrel {i-1}{\circ} )\times \{T\in Tilt(Q_{n-i+1})\mid (\dimvec T)_{1}=1,\ \delta (T)=n-i+1\}.\\
\end{array}\]
\end{lem}
 \begin{proof} Let $T\in \mathcal{T}_{1} $ s.t. $(\dimvec T)_{i}=1$, for a unique $i(\leq n-1).$
 By lemma\ref{510} and \ref{511} there exists a unique $j=j(T)(\geq i)$ s.t.
 $M(0,j)\mid T.$ Now let \[X(T)=\{L(a,b)\mid \ L(a,b)\mid T,\ b<i\}\] and
 \[Y(T)=\{N\in \ind kQ \mid \ N\mid T\}\setminus \left\{X(T)\cup \{M(0,j)\}\right\}.\] We define the maps  
\[\varphi_{T}:\ X(T)\longrightarrow \ind k(\circ \rightarrow \circ \rightarrow
 \cdots \rightarrow \stackrel {i-1}{\circ})\] and \[\psi _{T}:\ Y(T)\longrightarrow \ind kQ_{n-i+1},\] 
by
\[(\varphi _{T}(N))_{a}=(N)_{a}\ (1\leq a <i),\]
\[(\psi _{T}(N))_{a}=(N)_{a+i-1}\ (\mathrm{let}\ (n-i+1)^{\pm }+i-1=n^{\pm }).\]
Then 
 \[T\longmapsto \left(\bigoplus_{x\in X(T)} \varphi_{T}(x),\bigoplus_{y\in Y(T)} \psi_{T}(y)\bigoplus M(0,j(T)-i+1)\right)\]
induces a bijection between
\[\{T\in \mathcal{T}_{1} \mid (\dimvec T)_{i}=1\}\] and
\[Tilt(\circ \rightarrow \circ \rightarrow
 \cdots \rightarrow \stackrel {i-1}{\circ} )\times \{T\in Tilt(Q)\mid (\dimvec T)_{1}=1,\ \delta (T)=n-i+1\}.\]
\end{proof}
\ \ Let us define the following subsets of $\mathcal{T}_{1}$:
\[\begin{array}{lll}
\mathcal{A}_{\pm } & := & \left\{T\in \mathcal{T}_{1}\mid  \begin{array}{l}
 \mathrm{all}\ \mathrm{indecomposable\ direct\ summands\ of}\ T\ \mathrm{is\ not\ sincere} \\
\text{and}\ (\dimvec T)_{n^{\pm }}=1
\end{array}\right\}, \\\\
\mathcal{B}_{\pm } & := & \{T\in \mathcal{T}_{1}\mid  (\dimvec T)_{n^{\pm }}=1, \mathrm{there}\ \mathrm{exists}\ \mathrm{some}\ j\ s.t.M(0,j)\mid T\},\\
\mathcal{B}_{\pm }(j) & := & \{T\in \mathcal{B}_{\pm }\mid \ M(0,j)\mid T\},\\
\mathcal{C} & := & \{T\in \mathcal{T}_{1}\mid (\dimvec T)_{1}=1\},\\
\mathcal{C}(j) & := & \{T\in \mathcal{C}\mid \ M(0,j)\mid T\}.\\
\end{array} \]
\begin{thm} 
 \label{515}
  $(1):\mathcal{A}_{\pm }=\emptyset$.\\
  $(2):\mathcal{B}_{\pm }(j)\stackrel{1:1}{\longleftrightarrow} \{T^{'}\in Tilt(\circ \rightarrow 
 \cdots \rightarrow \stackrel{n}{\circ })\mid min\{j^{'}\mid L(j^{'},n-1)\mid T^{'}\}=j\}.$ In particular,
  \[\mathcal{B}_{\pm }\stackrel{1:1}{\longleftrightarrow}Tilt(\circ \rightarrow \cdots \rightarrow \stackrel{n}{\circ })\setminus \{T^{'}\in Tilt(\circ \rightarrow
 \cdots \rightarrow \stackrel{n}{\circ })\mid  L(0,n-1)\mid T^{'}\},\]
and we have \[\#\mathcal{B}_{\pm }=\frac{1}{n+1}{2n\choose n}-\frac{1}{n}{2(n-1)\choose n-1}.\]
  $(3):\mathcal{C}(j)\stackrel{1:1}{\longleftrightarrow} \{T^{'}\in Tilt(Q_{n-1})\mid j^{'}(T^{'})=j-1\}$\\
where 
\[j^{'}(T^{'})=\sup \{b\mid L^{+}(b,n-1)\ \mathrm{or}\ L^{-}(b,n-1)\ \mathrm{or}\ M(a,b)\mid T^{'} \mathrm{\ for\ some}\ a  \}.\]
 In particular,  
\[\mathcal{C}\stackrel{1:1}{\longleftrightarrow}Tilt(Q_{n-1}),\]
and we have 
\[\#\mathcal{C}=\frac{3n-4}{2n}{2(n-1)\choose n-1}.\]
\end{thm} 
\begin{proof} $(1):$ Suppose that there exists some $T\in \mathcal{A}_{+}.$ Then,
 by lemma\ref{511}, we have $L^{\pm }(0,n)\mid T.$ 
Now there exists 
 some indecomposable direct summand $N$ of $T$ not 
isomorphic to $L^{-}(0,n)$ s.t. 
$(\dimvec N)_{n^{-}}=1$. 

If $N=M(a,b)$ or $L^{-}(a,n)$ then $a=0.$ This is a contradiction because $L^{\pm }(0,n)\mid T$.
So $\mathcal{A}_{+}=\emptyset $ and similarly we have $A_{-}=\emptyset $.

$(2):$ Define the maps \[\begin{array}{rl}
\varphi :\{L(a,b)\mid 0\leq a< b\leq n-1\}\cup \{L^{-}(a,n)\mid 0\leq a\leq n-1\} &  \\
\longrightarrow & \ind  k(\circ \rightarrow \circ \rightarrow
 \cdots \rightarrow \stackrel {n}{\circ}) \\
\end{array}\]
 and \[\begin{array}{rc}
\psi :\ind k(\circ \rightarrow \circ \rightarrow
 \cdots \rightarrow \stackrel {n}{\circ}) &  \\
\longrightarrow & \{L(a,b)\mid 0\leq a< b\leq n-1\}\cup \{L^{-}(a,n)\mid 0\leq a\leq n-1\} \\
\end{array} \] by
\[\begin{array}{lll}
(\varphi (L))_{a} & = & \left\{ \begin{array}{ll}
L_{a} & \mathrm{if}\ 0\leq  a\leq n-1, \\
L_{n^{-}} & \mathrm{if }\ a=n,\\
\end{array}\right. \\
(\psi (L^{'}))_{a} & = & \left\{ \begin{array}{ll}
L^{'}_{a} & \mathrm{if }\ 0\leq a\leq n-1, \\
L^{'}_{n} & \mathrm{if }\ a=n^{-}, \\
0 & \mathrm{if }\ a=n^{+}. \\
\end{array} \right. \\
\end{array}\]
Then  $\varphi \circ \psi =1=\psi \circ \varphi $. For $T\in \mathcal{B}_{+}(j)$ and $T^{'}\in Tilt(\circ \rightarrow
 \cdots \rightarrow \stackrel{n}{\circ })$, define 
\[Z(T):=\{N\in \ind kQ\mid N\mid T,N\not\simeq M(0,j)\}\]
and
\[Y(T^{'}):=\{N\in \ind k(\circ \rightarrow 
 \cdots \rightarrow \stackrel{n}{\circ })\mid N\mid T^{'}\}.\]
Then it is easy to see that the maps induce a bijection
\[ \mathcal{B}_{+}(j)\stackrel{1:1}{\longleftrightarrow} \{T^{'}\in Tilt(\circ \rightarrow
 \cdots \rightarrow \stackrel{n}{\circ })\mid min\{j^{'}\mid L(j^{'},n-1)\mid T^{'}\}=j\}\]
 by
\[T\mapsto \bigoplus_{L\in Z(T)}\varphi(L).\]
 The inverse map is
\[T^{'}\mapsto \left(\bigoplus_{L^{'}\in Y(T^{'})}\psi (L^{'})\right)\oplus M(0,j).\]   
In fact, if $T\in \mathcal{B}_{+}(j)$ then all indecomposable direct summands of $T$ not isomorphic to $M(0,j)$  are either
\[ L(a,b)\;(a\geq j\ \text{or}\ b<j)\ \text{or } L^{-}(a,n)\;(a\leq j),\] which implies $L(j,n-1),L^{-}(j,n)\mid T$. It follows 
\[ min\{j^{'}\mid L(j^{'},n-1)\mid \bigoplus_{L\in Z(T)}\varphi(L)\}=j.\]
Conversely, if \[T^{'}\in \{T^{'}\in Tilt(\circ \rightarrow
 \cdots \rightarrow \stackrel{n}{\circ })\mid min\{j^{'}\mid L(j^{'},n-1)\mid T^{'}\}=j\}\] then $\left(\bigoplus_{L^{'}\in Y(T^{'})}\psi (L^{'})\right)\oplus M(0,j)\in 
\mathcal{B}_{+}(j).$
 
$(3):$ Define the maps
 \[\varphi :\{N\in \ind kQ_{n}\mid (\dimvec N)_{1}=0\}\longrightarrow \ind kQ_{n-1}\]
 and 
 \[\psi :\ind kQ_{n-1}\longrightarrow \{N\in \ind kQ_{n}\mid (\dimvec N)_{1}=0\}\] 
 by the obvious way. Then $\varphi \circ \psi =1=\psi \circ \varphi $. For $T\in \mathcal{C}(j)$ and $T^{'}\in Tilt(Q_{n-1})$, define   
\[Z(T):= \{N\in \ind kQ_{n}\mid N\mid T,N\not\simeq M(0,j)\}\]
and
\[Y(T^{'}):=\{N\in \ind kQ_{n-1}\mid N\mid T^{'}\}.\]
Then they induce a bijection 
\[\mathcal{C}(j)\stackrel{1:1}{\longleftrightarrow } \{T^{'}\in Tilt(Q_{n-1})\mid j^{'}(T^{'})=j-1\}\]
by \[T\longmapsto
 \bigoplus_{N\in Z(T)} \varphi (N).\]
 The inverse map is 
\[T^{'}\longmapsto 
 \left( \bigoplus_{N^{'}\in Y(T^{'})} \psi (N^{'})\right)\oplus M(0,j).\]
In fact, if $T\in \mathcal{C}(j)$ then \[\begin{array}{ll}
Z(T)\subset & \{L(a,b)\mid 1\leq a<b<j\}\cup \{L^{\pm }(b,n)\mid 1\leq b\leq j\} \\
 & \cup \{M(a,b) \mid 1\leq a<b\leq j\}. \\
\end{array}\]
It implies $M(1,j)\mid T$ and  $j^{'}\left( \bigoplus_{N\in Z(T)} \varphi (N) \right)=j-1.$
Conversely, if $j=j^{'}(T^{'})+1$ then \[\begin{array}{ll}
Y(T^{'})\subset & \{L_{n-1}(a,b)\mid b\leq j-2\ \text{or}\ a\geq j-1\}\cup \{L_{n-1}^{\pm }(b,n-1)\mid 0\leq b\leq j-1\} \\
 & \cup \{M_{n-1}(a,b) \mid 0\leq a<b\leq j-1\}, \\
\end{array}\] and \[\left(\dimvec \oplus_{N^{'}\in Y(T^{'})} \psi (N^{'})\right)_{a}\left\{ \begin{array}{ll}
\geq 1 & (a\geq 2) \\
=0 & (a=1). \\
\end{array}\right.\] It implies 
\[\left( \bigoplus_{N^{'}\in Y(T^{'})} \psi (N^{'})\right)\oplus M(0,j)\in \mathcal{C}(j).\] 
\end{proof}
\begin{cor} \[\# \mathcal{T}_{1}=3{2(n-1)\choose n-2}. \]
\end{cor}
\begin{proof} First we claim that \[\sum_{i=1}^{n}\frac{1}{i(n+1-i)}{2(i-1)\choose i-1}{2(n-i)\choose n-i}=\frac{1}{n+1}{2n\choose n}.\] This follows from the fact that \[Tilt(\circ \rightarrow \cdots \rightarrow \stackrel{n}{\circ})=\bigsqcup 
 \{T\in Tilt(\circ \rightarrow \cdots \rightarrow \stackrel{n}{\circ})\mid min\{i^{'}\mid L(i^{'},n)\mid T, i^{'}>0\}=i\}.\]
Thus, by lemma\ref{514} and theorem\ref{515}, $\# \mathcal{T}_{1}$ is equel to
 \begin{align*}
 & 2\left(\frac{1}{n+1}{2n\choose n}-\frac{1}{n}{2(n-1)\choose n-1}\right)+\sum_{i=1}^{n-1}\frac{3(n-i)-1}{2i(n-i+1)}{2(i-1)\choose i-1}{2(n-i)\choose n-i} \\\\
= & 2\left\{\left(\frac{1}{n+1}{2n\choose n}-\frac{1}{n}{2(n-1)\choose n-1}\right)-\sum_{i=1}^{n-1} \frac{1}{i(n-i+1)}{2(i-1)\choose i-1}{2(n-i)\choose n-i}\right\}\\\\
 & \ \ \ \ \ \ \ \ \ \ \ \ \ \ \ \ \   +\sum_{i=1}^{n-1} \frac{3}{2i}{2(i-1)\choose i-1}{2(n-i)\choose n-i} \\\\
= & \frac{3}{2} \sum_{i=1}^{n-1} \frac{1}{i}{2(i-1)\choose i-1}{2(n-i)\choose n-i}.\\
\end{align*}
Now let 
\[a_{n}=\sum_{i=1}^{n}\frac{1}{i}{2(i-1)\choose i-1}{2(n+1-i)\choose n+1-i}\]
 and 
\[f(X)=\left(\sum_{i=1}^{n} \frac{1}{i}{2(i-1)\choose i-1}X^{i} \right)^{2}.\]
Then the coefficient  of $X^{n+1}$ in $f^{'}(X)$ is equal to \[2a_{n}-2{2n\choose n}.\]
 On the other hand, the coefficient of $X^{n+2}$ in $f(X)$ is equal to
\begin{align*}
  & \sum_{i=1}^{n+1}\frac{1}{i(n-i+2)}{2(i-1)\choose i-1}{2(n-i+1)\choose n-i+1}-\frac{2}{n+1}{2n\choose n} \\\\
= & \frac{1}{n+2}{2(n+1)\choose n+1}-\frac{2}{n+1}{2n\choose n}.\\
\end{align*}
 So \[2a_{n} = {2(n+1)\choose n+1}-\frac{2}{n+1}{2n\choose n} = 4{2n\choose n-1}.\]
 We conclude that
 \[\# \mathcal{T}_{1}  =  \frac{3}{2}a_{n-1} = 3{2(n-1)\choose n-2}.\]
\end{proof}
\begin{cor}
 \label{517}
  We have \[\# \mathcal{T}_{0}=\frac{3(n-1)}{n+1}{2(n-1)\choose n-2}.\]
\end{cor}
\begin{proof} In fact, \begin{align*}
 \# \mathcal{T}_{0} & = \frac{3n-1}{2(n+1)}{2n\choose n}-3{2(n-1)\choose n-2}-\frac{1}{n}{2(n-1)\choose n-1} \\\\
  & =  \frac{3(n-1)}{n+1}{2(n-1)\choose n-2}. \\
\end{align*}
\end{proof}
\begin{thm} \[\# \vec{\mathcal{K}}(Q)_{1}=(3n-1){2(n-1)\choose n-2}.\]
\end{thm}
\begin{proof} In fact, $\# \vec{\mathcal{K}}(Q)_{1}$ is equal to
\begin{align*}
  & \frac{1}{2}\left\{ \frac{n-1}{n-1}{2(n-1)\choose n-2}+3n{2(n-1)\choose n-2}+3(n-1){2(n-1)\choose n-2}\right\} \\\\
= & (3n-1){2(n-1)\choose n-2}.\\
\end{align*}
\end{proof}

\begin{exmp} (n=3)

\unitlength 0.1in
\begin{picture}( 46.0000, 22.7000)( 10.0000,-24.3000)
%
\special{pn 8}%
\special{ar 1800 410 50 50  0.0000000 6.2831853}%
%
\special{pn 8}%
\special{ar 1180 800 52 52  0.0000000 6.2831853}%
%
\special{pn 8}%
\special{ar 2370 800 54 54  0.0000000 6.2831853}%
%
\special{pn 8}%
\special{ar 1180 1200 52 52  0.0000000 6.2831853}%
%
\special{pn 8}%
\special{ar 1800 1210 50 50  0.0000000 6.2831853}%
%
\special{pn 8}%
\special{ar 2370 1200 52 52  0.0000000 6.2831853}%
%
\special{pn 8}%
\special{ar 1500 1590 52 52  0.0000000 6.2831853}%
%
\special{pn 8}%
\special{ar 2090 1590 52 52  0.0000000 6.2831853}%
%
\special{pn 8}%
\special{ar 1800 1990 50 50  0.0000000 6.2831853}%
%
\special{pn 8}%
\special{ar 1800 2380 50 50  0.0000000 6.2831853}%
%
\special{pn 8}%
\special{pa 1740 450}%
\special{pa 1250 760}%
\special{fp}%
\special{sh 1}%
\special{pa 1250 760}%
\special{pa 1318 742}%
\special{pa 1296 732}%
\special{pa 1296 708}%
\special{pa 1250 760}%
\special{fp}%
%
\special{pn 8}%
\special{pa 1860 450}%
\special{pa 2300 760}%
\special{fp}%
\special{sh 1}%
\special{pa 2300 760}%
\special{pa 2258 706}%
\special{pa 2256 730}%
\special{pa 2234 738}%
\special{pa 2300 760}%
\special{fp}%
%
\special{pn 8}%
\special{pa 1800 470}%
\special{pa 1810 1130}%
\special{fp}%
\special{sh 1}%
\special{pa 1810 1130}%
\special{pa 1830 1064}%
\special{pa 1810 1078}%
\special{pa 1790 1064}%
\special{pa 1810 1130}%
\special{fp}%
%
\special{pn 8}%
\special{pa 1170 860}%
\special{pa 1180 1150}%
\special{fp}%
\special{sh 1}%
\special{pa 1180 1150}%
\special{pa 1198 1084}%
\special{pa 1178 1098}%
\special{pa 1158 1084}%
\special{pa 1180 1150}%
\special{fp}%
%
\special{pn 8}%
\special{pa 2370 860}%
\special{pa 2370 1140}%
\special{fp}%
\special{sh 1}%
\special{pa 2370 1140}%
\special{pa 2390 1074}%
\special{pa 2370 1088}%
\special{pa 2350 1074}%
\special{pa 2370 1140}%
\special{fp}%
%
\special{pn 8}%
\special{pa 1760 1250}%
\special{pa 1540 1540}%
\special{fp}%
\special{sh 1}%
\special{pa 1540 1540}%
\special{pa 1596 1500}%
\special{pa 1572 1498}%
\special{pa 1564 1476}%
\special{pa 1540 1540}%
\special{fp}%
%
\special{pn 8}%
\special{pa 1220 1250}%
\special{pa 1440 1540}%
\special{fp}%
\special{sh 1}%
\special{pa 1440 1540}%
\special{pa 1416 1476}%
\special{pa 1408 1498}%
\special{pa 1384 1500}%
\special{pa 1440 1540}%
\special{fp}%
%
\special{pn 8}%
\special{pa 1850 1250}%
\special{pa 2050 1540}%
\special{fp}%
\special{sh 1}%
\special{pa 2050 1540}%
\special{pa 2030 1474}%
\special{pa 2020 1496}%
\special{pa 1996 1496}%
\special{pa 2050 1540}%
\special{fp}%
%
\special{pn 8}%
\special{pa 2330 1250}%
\special{pa 2130 1550}%
\special{fp}%
\special{sh 1}%
\special{pa 2130 1550}%
\special{pa 2184 1506}%
\special{pa 2160 1506}%
\special{pa 2150 1484}%
\special{pa 2130 1550}%
\special{fp}%
%
\special{pn 8}%
\special{pa 1540 1640}%
\special{pa 1760 1930}%
\special{fp}%
\special{sh 1}%
\special{pa 1760 1930}%
\special{pa 1736 1866}%
\special{pa 1728 1888}%
\special{pa 1704 1890}%
\special{pa 1760 1930}%
\special{fp}%
%
\special{pn 8}%
\special{pa 2060 1630}%
\special{pa 1840 1930}%
\special{fp}%
\special{sh 1}%
\special{pa 1840 1930}%
\special{pa 1896 1888}%
\special{pa 1872 1888}%
\special{pa 1864 1864}%
\special{pa 1840 1930}%
\special{fp}%
%
\special{pn 8}%
\special{pa 1800 2060}%
\special{pa 1810 2330}%
\special{fp}%
\special{sh 1}%
\special{pa 1810 2330}%
\special{pa 1828 2264}%
\special{pa 1808 2278}%
\special{pa 1788 2264}%
\special{pa 1810 2330}%
\special{fp}%
\put(11.3000,-7.3000){\makebox(0,0)[lb]{$1$}}%
\put(23.4000,-7.3000){\makebox(0,0)[lb]{$2$}}%
\put(10.0000,-13.0000){\makebox(0,0)[lb]{$3$}}%
\put(24.5000,-12.7000){\makebox(0,0)[lb]{$4$}}%
\put(18.9000,-20.6000){\makebox(0,0)[lb]{$5$}}%
\put(16.6000,-26.0000){\makebox(0,0)[lb]{$6$}}%
%
\special{pn 8}%
\special{pa 1240 820}%
\special{pa 1420 950}%
\special{fp}%
\special{sh 1}%
\special{pa 1420 950}%
\special{pa 1378 896}%
\special{pa 1378 920}%
\special{pa 1354 928}%
\special{pa 1420 950}%
\special{fp}%
%
\special{pn 8}%
\special{pa 1260 1220}%
\special{pa 1430 1350}%
\special{fp}%
\special{sh 1}%
\special{pa 1430 1350}%
\special{pa 1390 1294}%
\special{pa 1388 1318}%
\special{pa 1366 1326}%
\special{pa 1430 1350}%
\special{fp}%
%
\special{pn 8}%
\special{pa 2320 820}%
\special{pa 2120 960}%
\special{fp}%
\special{sh 1}%
\special{pa 2120 960}%
\special{pa 2186 938}%
\special{pa 2164 930}%
\special{pa 2164 906}%
\special{pa 2120 960}%
\special{fp}%
%
\special{pn 8}%
\special{pa 2320 1220}%
\special{pa 2130 1340}%
\special{fp}%
\special{sh 1}%
\special{pa 2130 1340}%
\special{pa 2198 1322}%
\special{pa 2176 1312}%
\special{pa 2176 1288}%
\special{pa 2130 1340}%
\special{fp}%
%
\special{pn 8}%
\special{pa 1790 1710}%
\special{pa 1810 1930}%
\special{fp}%
\special{sh 1}%
\special{pa 1810 1930}%
\special{pa 1824 1862}%
\special{pa 1806 1878}%
\special{pa 1784 1866}%
\special{pa 1810 1930}%
\special{fp}%
%
\special{pn 8}%
\special{pa 1550 2170}%
\special{pa 1730 2340}%
\special{fp}%
\special{sh 1}%
\special{pa 1730 2340}%
\special{pa 1696 2280}%
\special{pa 1692 2304}%
\special{pa 1668 2310}%
\special{pa 1730 2340}%
\special{fp}%
%
\special{pn 8}%
\special{ar 4810 400 52 52  0.0000000 6.2831853}%
%
\special{pn 8}%
\special{ar 4810 670 50 50  0.0000000 6.2831853}%
%
\special{pn 8}%
\special{ar 4360 980 52 52  0.0000000 6.2831853}%
%
\special{pn 8}%
\special{ar 5240 980 50 50  0.0000000 6.2831853}%
%
\special{pn 8}%
\special{ar 4810 980 50 50  0.0000000 6.2831853}%
%
\special{pn 8}%
\special{ar 4350 1370 52 52  0.0000000 6.2831853}%
%
\special{pn 8}%
\special{ar 4810 1380 54 54  0.0000000 6.2831853}%
%
\special{pn 8}%
\special{ar 5240 1380 50 50  0.0000000 6.2831853}%
%
\special{pn 8}%
\special{ar 4800 1810 50 50  0.0000000 6.2831853}%
%
\special{pn 8}%
\special{ar 4800 2140 50 50  0.0000000 6.2831853}%
%
\special{pn 8}%
\special{pa 4810 460}%
\special{pa 4820 610}%
\special{fp}%
\special{sh 1}%
\special{pa 4820 610}%
\special{pa 4836 542}%
\special{pa 4816 558}%
\special{pa 4796 546}%
\special{pa 4820 610}%
\special{fp}%
%
\special{pn 8}%
\special{pa 4760 700}%
\special{pa 4410 940}%
\special{fp}%
\special{sh 1}%
\special{pa 4410 940}%
\special{pa 4476 920}%
\special{pa 4454 910}%
\special{pa 4454 886}%
\special{pa 4410 940}%
\special{fp}%
%
\special{pn 8}%
\special{pa 4870 700}%
\special{pa 5190 940}%
\special{fp}%
\special{sh 1}%
\special{pa 5190 940}%
\special{pa 5150 884}%
\special{pa 5148 908}%
\special{pa 5126 916}%
\special{pa 5190 940}%
\special{fp}%
%
\special{pn 8}%
\special{pa 4810 730}%
\special{pa 4810 920}%
\special{fp}%
\special{sh 1}%
\special{pa 4810 920}%
\special{pa 4830 854}%
\special{pa 4810 868}%
\special{pa 4790 854}%
\special{pa 4810 920}%
\special{fp}%
%
\special{pn 8}%
\special{pa 4360 1040}%
\special{pa 4360 1320}%
\special{fp}%
\special{sh 1}%
\special{pa 4360 1320}%
\special{pa 4380 1254}%
\special{pa 4360 1268}%
\special{pa 4340 1254}%
\special{pa 4360 1320}%
\special{fp}%
%
\special{pn 8}%
\special{pa 4770 1050}%
\special{pa 4410 1300}%
\special{fp}%
\special{sh 1}%
\special{pa 4410 1300}%
\special{pa 4476 1278}%
\special{pa 4454 1270}%
\special{pa 4454 1246}%
\special{pa 4410 1300}%
\special{fp}%
%
\special{pn 8}%
\special{pa 4860 1040}%
\special{pa 5190 1300}%
\special{fp}%
\special{sh 1}%
\special{pa 5190 1300}%
\special{pa 5150 1244}%
\special{pa 5148 1268}%
\special{pa 5126 1274}%
\special{pa 5190 1300}%
\special{fp}%
%
\special{pn 8}%
\special{pa 5240 1040}%
\special{pa 5240 1330}%
\special{fp}%
\special{sh 1}%
\special{pa 5240 1330}%
\special{pa 5260 1264}%
\special{pa 5240 1278}%
\special{pa 5220 1264}%
\special{pa 5240 1330}%
\special{fp}%
%
\special{pn 8}%
\special{pa 4410 1030}%
\special{pa 4750 1320}%
\special{fp}%
\special{sh 1}%
\special{pa 4750 1320}%
\special{pa 4712 1262}%
\special{pa 4710 1286}%
\special{pa 4686 1292}%
\special{pa 4750 1320}%
\special{fp}%
%
\special{pn 8}%
\special{pa 5200 1030}%
\special{pa 4870 1310}%
\special{fp}%
\special{sh 1}%
\special{pa 4870 1310}%
\special{pa 4934 1282}%
\special{pa 4912 1276}%
\special{pa 4908 1252}%
\special{pa 4870 1310}%
\special{fp}%
%
\special{pn 8}%
\special{pa 4810 1440}%
\special{pa 4810 1740}%
\special{fp}%
\special{sh 1}%
\special{pa 4810 1740}%
\special{pa 4830 1674}%
\special{pa 4810 1688}%
\special{pa 4790 1674}%
\special{pa 4810 1740}%
\special{fp}%
%
\special{pn 8}%
\special{pa 4390 1430}%
\special{pa 4750 1750}%
\special{fp}%
\special{sh 1}%
\special{pa 4750 1750}%
\special{pa 4714 1692}%
\special{pa 4710 1716}%
\special{pa 4688 1722}%
\special{pa 4750 1750}%
\special{fp}%
%
\special{pn 8}%
\special{pa 5210 1440}%
\special{pa 4860 1750}%
\special{fp}%
\special{sh 1}%
\special{pa 4860 1750}%
\special{pa 4924 1722}%
\special{pa 4900 1716}%
\special{pa 4898 1692}%
\special{pa 4860 1750}%
\special{fp}%
%
\special{pn 8}%
\special{pa 4810 1890}%
\special{pa 4800 2070}%
\special{fp}%
\special{sh 1}%
\special{pa 4800 2070}%
\special{pa 4824 2006}%
\special{pa 4804 2018}%
\special{pa 4784 2002}%
\special{pa 4800 2070}%
\special{fp}%
%
\special{pn 8}%
\special{pa 4800 2210}%
\special{pa 4810 2390}%
\special{fp}%
\special{sh 1}%
\special{pa 4810 2390}%
\special{pa 4826 2322}%
\special{pa 4808 2338}%
\special{pa 4786 2326}%
\special{pa 4810 2390}%
\special{fp}%
\put(46.8000,-25.8000){\makebox(0,0)[lb]{$(6)$}}%
%
\special{pn 8}%
\special{pa 4770 1430}%
\special{pa 4460 1870}%
\special{fp}%
\special{sh 1}%
\special{pa 4460 1870}%
\special{pa 4516 1828}%
\special{pa 4492 1826}%
\special{pa 4482 1804}%
\special{pa 4460 1870}%
\special{fp}%
\put(43.2000,-20.4000){\makebox(0,0)[lb]{$(5)$}}%
%
\special{pn 8}%
\special{pa 4390 280}%
\special{pa 4750 380}%
\special{fp}%
\special{sh 1}%
\special{pa 4750 380}%
\special{pa 4692 344}%
\special{pa 4700 366}%
\special{pa 4680 382}%
\special{pa 4750 380}%
\special{fp}%
%
\special{pn 8}%
\special{pa 5210 280}%
\special{pa 4870 370}%
\special{fp}%
\special{sh 1}%
\special{pa 4870 370}%
\special{pa 4940 372}%
\special{pa 4922 356}%
\special{pa 4930 334}%
\special{pa 4870 370}%
\special{fp}%
%
\special{pn 8}%
\special{pa 4000 770}%
\special{pa 4280 960}%
\special{fp}%
\special{sh 1}%
\special{pa 4280 960}%
\special{pa 4236 906}%
\special{pa 4236 930}%
\special{pa 4214 940}%
\special{pa 4280 960}%
\special{fp}%
%
\special{pn 8}%
\special{pa 5560 710}%
\special{pa 5290 920}%
\special{fp}%
\special{sh 1}%
\special{pa 5290 920}%
\special{pa 5356 896}%
\special{pa 5332 888}%
\special{pa 5330 864}%
\special{pa 5290 920}%
\special{fp}%
\put(41.3000,-3.4000){\makebox(0,0)[lb]{$(1)$}}%
\put(52.5000,-3.3000){\makebox(0,0)[lb]{$(2)$}}%
\put(37.3000,-8.4000){\makebox(0,0)[lb]{$(3)$}}%
\put(56.0000,-8.2000){\makebox(0,0)[lb]{$(4)$}}%
\end{picture}%

\end{exmp}
\section*{Acknowledgment}
The author would like to express his gratitude to Professor Susumu Ariki for his mathematical supports, careful reading of this paper and warm encouragements.
The author also thanks Professor Syu Kato for his mathematical advices and warm encouragements.


\begin{thebibliography}{99}
\bibitem{ARS}M.~Auslander, I.~Reiten and S.~Smal\o, Representation theory of artin algebras, Cambridge University Press, 1995.
\bibitem{ASS}I.~Assem, D.~Simson and A.~Skowro\'{n}ski, Elements of the representation theory of associative algebras Vol.~\textbf{1}, London Mathematical Society Student Texts \textbf{65}, Cambridge University Press, 2006.
\bibitem{BMRRT}A.~B.~Buan, R.~Marsh, M.~Reineke, I.~Reiten and G.~Todorov, Tilting theory and cluster combinatorics, Adv. Math. \textbf{204} (2006), no.2, 572-618.  
\bibitem{CCS}P.~Caldero, F.~Chapoton and R.~Schiffler, Quiver with relations arising from clusters ($A_{n}$ case), Trans. Amer. Math. Soc. \textbf{358} (2006), no.3, 1347-1364. 
\bibitem{CHU}F.~Coelho, D.~Happel and L.~Unger, Complements to partial tilting modules, J.~Algebra \textbf{170} (1994), no.3, 184-205.
\bibitem{FZ}S.~Fomin and A.~Zelevinsky, $Y$-systems and generalized associahedra, Ann. of Math. (2) \textbf{158}, (2003), no.3, 977-1018.
\bibitem{HR} D.~Happel and C.~M.~Ringel, Tilted algebras, Trans. Amer. Math. Soc. \textbf{274} (1982), no.2, 399-443.
\bibitem{HU1}D.~Happel and L.~Unger, On a partial order of tilting modules, Algebr. Represent. Theory \textbf{8} (2005), no.2, 147-156. 
\bibitem{HU2}D.~Happel and L.~Unger, On the quiver of tilting modules, J. Algebra \textbf{284} (2005), no.2, 857-868.
\bibitem{HU3}D.~Happel and L.~Unger, Reconstruction of path algebras from their posets of tilting modules, Trans. Amer . Math. Soc \textbf{361} (2009), no.7, 3633-3660.
\bibitem{L}S.~Ladkani, Universal derived equivalences of posets of tilting modules, arXiv:0708.1287v1.
\bibitem{MRZ}R.~Marsh, M.~Reineke and A.~Zelevinsky, Generalized associahedra via quiver representations, Trans. Amer. Math. Soc \textbf{355} (2003), no.10, 4171-4186. 
\bibitem{R}I.~Reiten, Tilting theory and homologically finite subcategories, Handbook of tilting theory, L.Angeleri H\"{u}gel, D.Happel, H.Krause, eds., London Mathematical Society Lecture Note Series \textbf{332}, Cambridge University Press, 2007. 
\bibitem{RS}C.~Riedtmann and A.~Schofield, On a simplicial complex associated with tilting modules, Comment. Math. Helv \textbf{66} (1991), no.1, 70-78.
\bibitem{S}R.~Schiffler, A geometric model for cluster categories of type $D_{n}$, J. Algebraic Combin. \textbf{27} (2008), no.1, 1-21. 
\bibitem{U}L.~Unger, Combinatorial aspects of the set of tilting modules, Handbook of tilting theory, L.Angeleri H\"{u}gel, D.Happel, H.Krause, eds., London Mathematical Society Lecture Note Series \textbf{332}, Cambridge University Press, 2007.
\end{thebibliography}
\end{document}